\documentclass{amsart}
\usepackage{graphicx}
\usepackage{tikz}
\usepackage{mathrsfs}
\usepackage{ tipa }
\usetikzlibrary{arrows}
\usetikzlibrary{matrix}
\usepackage{amsmath}
\usepackage{mathrsfs}
\usepackage{amsthm}
\usepackage{amssymb}
\usepackage{enumerate}
\usepackage{url}

\usepackage{graphicx}
\usepackage[all]{xypic}
\usepackage{setspace}
\usepackage[all,cmtip]{xy}
\usepackage{fancyhdr}

\newcommand{\Q}{\mathbb{Q}}

\newcommand{\Z}{\mathbb{Z}}
\renewcommand{\Im}{\mathrm{Im}\,}
\def\Frac{\displaystyle\frac}
\DeclareMathOperator{\Ker}{Ker}

\newtheorem{theorem}{Theorem}[section]
\newtheorem{lemma}[theorem]{Lemma}

\theoremstyle{definition}
\newtheorem{definition}[theorem]{Definition}

\theoremstyle{remark}

\numberwithin{equation}{section}

\begin{document}

\title{A Fox-Milnor Theorem for Knots in a Thickened Surface}

\author{James Kreinbihl}
\address{Department of Mathematics and Computer Science, Wesleyan University, Middletown, Connecticut 06459}
\curraddr{Department of Mathematics,
Trinity College, Hartford, Connecticut 06106}
\email{james.kreinbihl@trincoll.edu}


%    General info


\begin{abstract}
A knot in a thickened surface $K$ is a smooth embedding $K:S^1 \rightarrow \Sigma \times [0,1]$, where $\Sigma$ is a closed, connected, orientable surface. There is a bijective correspondence between knots in $S^2 \times [0,1]$ and knots in $S^3$, so one can view the study of knots in thickened surfaces as an extension of classical knot theory. An immediate question is if other classical definitions, concepts, and results extend or generalize to the study of knots in a thickened surface. One such famous result is the Fox Milnor Theorem, which relates the Alexander polynomials of concordant knots. We prove a Fox Milnor Theorem for concordant knots in a thickened surface by using Milnor torsion.
\end{abstract}

\maketitle

%%%%%%%%%%%%%%%%%%%%%%%%%%%%%%%%%%%%%%%%%%%%%%%%%%%%%%%%%%%%%%%%%%%%%%%%

%%%%%%%%%%%%%%%%%%%%%%%%%%%%%%%%%%%%%%%%%%%%%%%%%%%%%%%%%%%%%%%%%%%%%%%%

\section{Introduction}

A \textit{knot in a thickened surface} $K$ is a smooth embedding $K: S^1 \rightarrow \Sigma \times I$ where $\Sigma$ is a closed, connected, orientable surface. An example, called the virtual trefoil, is depicted below.
\begin{figure}[h]\begin{centering}
\includegraphics[scale=.15]{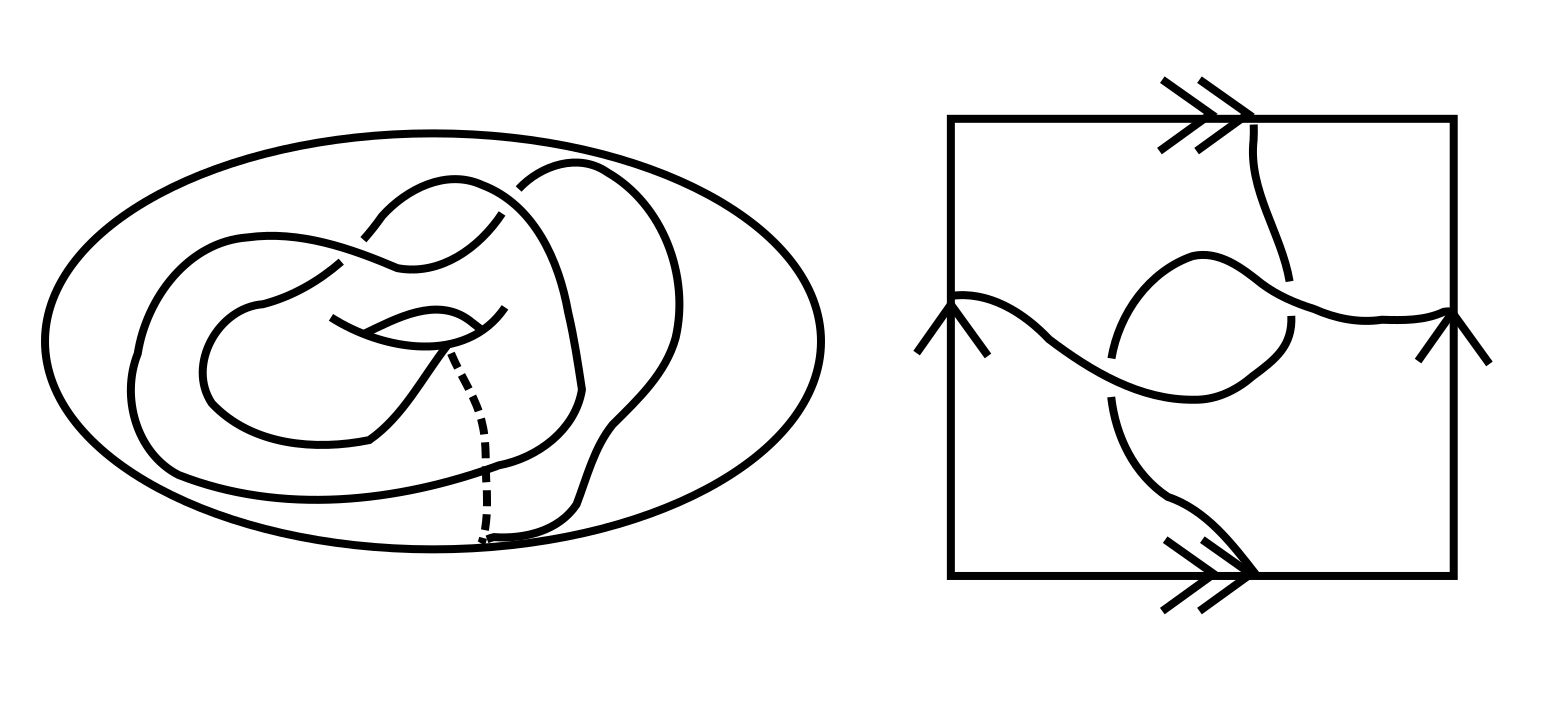}
\caption{A diagram of the virtual trefoil}
\label{vt}
\end{centering}
\end{figure}
\noindent Two knots in a thickened surface $K_0, K_1$ are \textit{equivalent} if there exists an orientation-preserving diffeomorphism 
\[
f: \left(\Sigma \times I, \Sigma \times \left\{0\right\}\right)\rightarrow \left(\Sigma \times I, \Sigma \times \left\{0\right\}\right)
\]
such that $f \circ K_0 = K_1$. A knot in a thickened surface is trivial if it bounds a smoothly embedded disk in $\Sigma \times I$. It is important to note that this definition does not allow for stablization or destabilization, the addition or reduction of genus to the surface $\Sigma$. There is a bijective correspondence between knots in $S^2 \times I$ and knots in $S^3$, so we can interpret the study of knots in a thickened surface as a generalization of classical knot theory. As such, one hopes that many of the operations and results of classical knot theory can be extended to the study of knots in a thickened surface. One such famous result is the Fox-Milnor Theorem which relates the Alexander polynomials of concordant knots. 

The goal of this paper is to prove a Fox Milnor Theorem for concordant knots in a thickened surface. First we summarize and reinterpret a definition of an Alexander polynomial of a knot in a thickened surface, proposed by Carter, Silver and Williams in \cite{LTS}. This discussion leads to a slightly altered definition of an Alexander polynomial. Then, we provide a definition of concordance of knots in a thickened surface which is more restrictive than the definition of virtual concordance proposed in previous work \cite{stableequiv, Cobord}. The proof of our main result follows a strategy similar to the proof of the Fox Milnor Theorem. Thus, we investigate the concordance complement and its boundary. We then relate the Alexander polynomial of a knot in a thickened surface to the Milnor torsion of a certain pair of spaces. Lastly, we prove our Fox Milnor Theorem by using results about Milnor torsion \cite{Turaev} and a duality theorem for Milnor torsion. 

\section{An Alexander Polynomial for Knots in a Thickened Surface}

In \cite{LTS}, Carter, Silver, and Williams propose a definition for an Alexander polynomial of a knot in a thickened surface. We briefly outline their construction. Let $\Sigma$ be a closed, connected, orientable surface with $g\left(\Sigma\right) \geq 1$ and let $\Gamma = \pi_1\left(\Sigma\right)$. 
The covering space of $\Sigma$ corresponding to the trivial subgroup of $\Gamma$ is the universal cover of $\Sigma$, which we denote $p: \widetilde{\Sigma} \rightarrow \Sigma$. The group of deck transformations of $\widetilde{\Sigma}$ is isomorphic to $\Gamma$. Since $p: \widetilde{\Sigma} \rightarrow \Sigma$ and $\text{id}:I \rightarrow I$ are both covering maps, it follows that the product map $p \times \text{id}: \widetilde{\Sigma} \times I \rightarrow \Sigma \times I$ is also a covering map. We will use the symbol $p$ to refer to this product covering map. A deck transformation $h: \widetilde{\Sigma} \rightarrow \widetilde{\Sigma}$ can be extended to a deck transformation $h \times \text{id} :\widetilde{\Sigma} \times I \rightarrow \widetilde{\Sigma} \times I$. Conversely, a deck transformation $f : \widetilde{\Sigma} \times I \rightarrow \widetilde{\Sigma} \times I$ can be restricted to a deck transformation $f:\widetilde{\Sigma} \rightarrow \widetilde{\Sigma}$. Since the map $\pi_1\left(\Sigma \right) \rightarrow \pi_1\left(\Sigma \times I\right)$ induced by inclusion is an isomorphism, it follows that the group of deck transformations of $\Sigma \times I$ is isomorphic to $\Gamma$. Let $K$ be a knot in the thickened surface $\Sigma \times I$ and let $\widetilde{K}=p^{-1}(K)$. Let $X= \Sigma \times I \backslash K$ denote the knot complement and let $\widetilde{X}=\widetilde{\Sigma} \times I \backslash \widetilde{K}$. Since $X$ is a subspace of $\Sigma \times I$ and $p^{-1}\left(X\right)=\widetilde{X}$, it follows that the restriction of $p:\widetilde{\Sigma} \times I \rightarrow \Sigma \times I$ to $p:\widetilde{X} \rightarrow X$ is a covering map. A deck transformation $h: \widetilde{\Sigma} \times I \rightarrow \widetilde{\Sigma} \times I$ can be restricted to deck transformation $h: \widetilde{X} \rightarrow \widetilde{X}$. Thus, the group of deck transformations of $\widetilde{X}$ is still isomorphic to $\Gamma$. The covering group of the knot is $\widetilde{\pi_K}= \pi_1\left(\widetilde{X}\right)$. In this discussion, we choose and fix a base point in $\widetilde{\Sigma} \times \{1\} \subset \widetilde{X}$. The covering group is a knot invariant. If $K_1$ and $K_2$ are equivalent knots in a thickened surface $\Sigma \times I$,then there exists a homeomorphism $f: \left(\Sigma \times I, \Sigma \left\{0\right\}\right)\rightarrow \left(\Sigma \times I, \Sigma \left\{0\right\}\right)$ such that $f \circ K_1 = K_2$. The map $f$ lifts to
\[
\xymatrix{
  \widetilde{\Sigma} \times I \ar[d]^{p} \ar[r]^{\widetilde{f}}  &   \widetilde{\Sigma} \times I \ar[d]^{p} \\
  \Sigma \times I \ar[r]^{f}  & \Sigma \times I
}
\]
Note that 
\[
\widetilde{f}\left(\widetilde{K_1}\right)=\widetilde{f}\left(p^{-1}(K_1)\right)=p^{-1}\left(f(K_1)\right)=p^{-1}(K_2)=\widetilde{K_2}.
\]
Therefore, $\widetilde{f}$ restricts to a homeomorphism $\widetilde{f}:\widetilde{X_1} \rightarrow \widetilde{X_2}$, where $\widetilde{X_i}=\widetilde{\Sigma} \times I \backslash \widetilde{K_i}$ for $i=1,2$. This homeomorphism induces an isomorphism of the covering groups. In their paper, Carter, Silver, and Williams state that one can compute a presentation of the covering group $\widetilde{\pi_K}= \pi_1\left(\widetilde{X}\right)$ using Wirtinger's Algorithm and that the generators of this presentation are meridians of the lift of knot in $\widetilde{\Sigma} \times I$. 

Before discussing the Carter, Silver, Williams definition of an Alexander polynomial, we recall the definition of the classical Alexander polynomial. Let $K$ be a knot in $S^3$ and let $X=S^3-K$. The Alexander module of the knot is $H_1(X;\Z[H_1(X)])$ and the Alexander polynomial of the knot is the order of the Alexander module. We now demonstrate an equivalent definition of the Alexander module and Alexander polynomial. Consider the abelianization homomorphism 
\[
\varphi: \pi_1(X) \rightarrow H_1(X) \cong \langle t \rangle.
\]
The covering space $p:\widetilde{X} \rightarrow X$ corresponding to $\Ker \varphi=\left[\pi_1(X), \pi_1(X)\right]= \pi_1(X)^{(1)}$ of $\pi_1(X)$ is the infinite cyclic cover of the knot complement. Note that the Alexander module $H_1(X;\Z[H_1(X)])$ and $H_1\left(\widetilde{X};\Z\right)$ are isomorphic as abelian groups. The covering map $p$ induces an isomorphism $\pi_1\left(\widetilde{X}\right) \cong p_*\left(\pi_1\left(\widetilde{X}\right)\right) = \Ker \varphi = \pi_1(X)^{(1)}$. Let $\pi_1(X)^{(2)}$ denote $\left[\pi_1(X)^{(1)},\pi_1(X)^{(1)}\right]$. Then, 
\begin{center}
\begin{tikzpicture}[->,>=stealth',auto,node distance=3cm,
  thick,main node/.style={font=\sffamily\Large\bfseries}]

  \node[main node] (1) {$\frac{\pi_1(X)^{(1)}}{\pi_1(X)^{(2)}}$};
  \node[main node] (2) [right of=1] {$\frac{\pi_1\left(\widetilde{X}\right)}{\pi_1\left(\widetilde{X}\right)^{(1)}}$};
  \node[main node] (3) [right of=2] {$H_1\left(\widetilde{X};\Z\right)$};

  \path[every node/.style={font=\sffamily\small}]
    (1) edge node [right] [above]{}(2)
    (2) edge node [right] [above]{}(3)
    (1) edge[bend left] node [right] [above]{$\psi$} (3);
\end{tikzpicture}
\end{center}
the homomorphism $\psi$ is an isomorphism of abelian groups. The isomorphism $\psi$ can be made into a $\Z[t,t^{-1}]$-module isomorphism. Thus, we can think of the Alexander module of a knot as $\Frac{\Ker \varphi}{[\Ker \varphi, \Ker \varphi]}$ with the $\Z[t,t^{1-}]$-module structure and the Alexander polynomial as the order of this module.

We now return to the case of a knot in a thickened surface. The definition of the Alexander polynomial proposed by Carter, Silver, and Williams resembles the second equivalent definition of classical Alexander polynomial. Define $\epsilon : \widetilde{\pi}_K \rightarrow \Z = \langle t \rangle$ to be the homomorphism mapping every meridian of the lift of $K$ to $t$. Let $M=\frac{\Ker \epsilon}{[\Ker \epsilon, \Ker \epsilon]}$ denote the abelianization of $\Ker \epsilon$. The abelian group $M$ can be given a $\Z[\Z \times \Gamma]$-module structure. We want to associate a Noetherian module to each knot so that we may compute elementary ideals of the module. Thus, define $\overline{M}=\frac{M}{M_0}$ where $M_0$ be the submodule of $M$ generated by all elements of the form $m^{\gamma}-m^{\eta}$ where $m \in M$ and $\gamma, \eta \in \Gamma$ such that $\gamma \eta^{-1} \in [\Gamma, \Gamma]$. One may consider $\overline{M}$ as an Alexander module of the knot. The Alexander polynomial is defined to be $\Delta_0(\overline{M})$ which is an element of the ring $\Z\left[\Z \times \frac{\Gamma}{[\Gamma,\Gamma]}\right]$.

This algebraic definition of the Alexander module can be described equivalently using homology with local coefficients. Consider the homomorphism $\epsilon: \pi_1\left(\widetilde{X}\right) \rightarrow \Z = \langle t \rangle$ defined previously. There is a covering space $p':\widetilde{X}' \rightarrow \widetilde{X}$ corresponding to the subgroup $\Ker \epsilon \leq \pi_1\left(\widetilde{X}\right)$ such that $\pi_1\left(\widetilde{X}'\right) \cong p_*'\left(\pi_1\left(\widetilde{X}'\right)\right) = \Ker \epsilon$. Therefore,
\begin{align*}
H_1\left(\widetilde{X}';\Z\right) & \cong \frac{\pi_1\left(\widetilde{X}'\right)}{\left[\pi_1\left(\widetilde{X}'\right),\pi_1\left(\widetilde{X}'\right)\right]}\\ & \cong \frac{p_*'\left(\pi_1\left(\widetilde{X}'\right)\right)}{\left[p_*'\left(\pi_1\left(\widetilde{X}'\right)\right),p_*'\left(\pi_1\left(\widetilde{X}'\right)\right)\right]}\\ &=\frac{\Ker \epsilon}{[\Ker \epsilon,\Ker \epsilon]} = M.
\end{align*}
The group $H_1\left(\widetilde{X}';\Z\right)$ has a $\Z[\Z\times \Gamma]$-module structure. The covering space $\widetilde{X}'$ is equivalent to the covering space of $X$ which corresponds to the kernel of the homomorphism $\phi:\pi_1(X) \rightarrow \Z \times \Gamma$. The spaces $X, \widetilde{X}$ and $\widetilde{X}'$ are related as below.  
\begin{center}
\begin{tikzpicture}[->,>=stealth',auto,node distance=3cm,
  thick,main node/.style={font=\sffamily\Large\bfseries}]

  \node[main node] (1) {$\widetilde{X}'$};
  \node[main node] (2) [right of=1] {$\widetilde{X}$};
  \node[main node] (3) [right of=2] {$X$};

  \path[every node/.style={font=\sffamily\small}]
    (1) edge node [right] [above]{$p'$}(2)
    (2) edge node [right] [above]{$p$}(3)
    (1) edge[bend left] node [right] [above]{$q$} (3);
\end{tikzpicture}
\end{center}
Therefore, $H_1\left(\widetilde{X}';\Z\right)=H_1\left(X;\Z[\Z\times \Gamma]\right)$ has a $\Z[\Z \times \Gamma]$-module structure with coefficient system defined by $\phi:\pi_1(X) \rightarrow \Z \times \Gamma$. In order to have a Noetherian module over the ring $\Z\left[\Z \times \frac{\Gamma}{[\Gamma,\Gamma]}\right]$, we previously considered a quotient module. We can interpret this quotient as a tensor product by using the following result.

\begin{lemma}\label{quotensor}
Let $G$ be a group. If $M$ is a right $\Z G$-module and $H \unlhd G$, then $M \otimes_{\Z G} \Z [\frac{G}{H}] \cong \frac{M}{M_0}$ where $M_0$ is the submodule generated by elements of the form $xr-xs$ for all $x \in M$ and $r^{-1}s \in H$.
\end{lemma}

Therefore, the Alexander module, as defined by Carter, Silver, and Williams can equivalently be defined as
\[
\frac{M}{M_0} \cong  H_1\left(X;\Z\left[\Z \times \Gamma \right]\right) \otimes_{\Z[\Z \times \Gamma]} \Z\left[\Z \times \frac{\Gamma}{[\Gamma,\Gamma]}\right].
\] 
Unfortunately, this definition suffers from a base point issue \cite{privcomm}. Recall that a base point in $\widetilde{\Sigma} \times \{1\} \subset \widetilde{X}$ is chosen. But, in their use of Fox calculus computations to compute a presentation matrix of the Alexander module, all 1-chains are treated as 1-cycles. Thus, their computations produce the order of the module
\[
H_1\left(\widetilde{X}',p'^{-1}(w);\Z\right)
\]
where $w$ is a fixed base point in $\widetilde{X}$ \cite{privcomm}. Clearly, the intention of this definition is to treat $\widetilde{X}$, in particular $\widetilde{\Sigma} \times \{1\} \subset \widetilde{X}$, as a connected space with a single 0-cell and to then consider the homology of the infinite cyclic cover of $\widetilde{X}$ as a $\Z[\Z\times \Gamma]$-module. However, in doing so, the action of $\Gamma$ on the Alexander module becomes undefined. In order to have an action of $\Gamma$, one must maintain the cell structure that is lifted from the base space $X=\Sigma \times I \backslash K$. That is, in order to have $\Gamma$ act by deck transformations on $\widetilde{X}=\widetilde{\Sigma} \times I \backslash \widetilde{K}$, one must use the cell structure of $\widetilde{X}$ which has a 0-cell for each element of the group $\Gamma$. With these considerations in mind, we propose the following definition.

\begin{definition}
Let $K$ be a knot in a thickened surface $\Sigma \times I$ and let $X=\Sigma \times I \backslash K$. The \textit{Alexander module of the knot in a thickened surface $K$} is $H_1\left(X,\Sigma \times \left\{0\right\};\Z\left[H_1(X)\right]\right)$.
\end{definition} 

\noindent In \cite{LTS}, it is shown that $H_1(X) \cong \Z \times\frac{\Gamma}{[\Gamma,\Gamma]}$, so this is still a $\Z\left[\Z \times \frac{\Gamma}{[\Gamma,\Gamma]}\right]$-module. Recall that if two knots $K_0, K_1$ in a thickened surface $\Sigma \times I$ are equivalent, then there exists a homeomorphism
\[
f: \left(\Sigma \times I, \Sigma \times \left\{0\right\}\right)\rightarrow \left(\Sigma \times I, \Sigma \times \left\{0\right\}\right)
\]
such that $f \circ K_0 = K_1$. Therefore, $f$ restricts to a homeomorphism  
\[
f: \left(X_0, \Sigma \times \left\{0\right\}\right)\rightarrow \left(X_1, \Sigma \times \left\{0\right\}\right)
\]
where $X_i=\Sigma \times I \backslash K_i$ for $i=0,1$. This homeomorphism induces an isomorphism of the Alexander modules. Since the Alexander module defined above is a knot invariant, we now define an Alexander polynomial of a knot in a thickened surface. 
\begin{definition}
Let $K$ be a knot in a thickened surface $\Sigma \times I$ and let $X=\Sigma \times I \backslash K$. The \textit{Alexander polynomial of the knot in a thickened surface $K$} is 
\[
\Delta(K) = \Delta_0\left(H_1\left(X,\Sigma \times \left\{0\right\};\Z\left[H_1(X)\right]\right)\right).
\]
\end{definition} 
\noindent While this definition of the Alexander polynomial ultimately differs from that of Carter, Silver, and Williams in \cite{LTS}, it is the polynomial that they compute in their paper. Furthermore, the Alexander module we define is still the homology of $\widetilde{X}'$, the infinite cyclic cover of $\widetilde{X}$. That is, we will show that 
\[
H_1(X,\Sigma \times \left\{0\right\};\Z[\Z \times \Gamma]) \otimes_{\Z[\Z \times \Gamma]} \Z\left[\Z \times \frac{\Gamma}{[\Gamma,\Gamma]}\right] \cong H_1\left(X,\Sigma \times \left\{0\right\};\Z\left[\Z \times \frac{\Gamma}{[\Gamma,\Gamma]}\right]\right).
\]
We make use of the following result in our proof.

\begin{theorem}\cite{mccleary}
Let $R$ be a ring, $A_{\ast}$ be a positive complex of flat right $R$-modules, and let $C_{\ast}$ be a positive complex of left $R$-modules. Then,
\[
\text{E}_{p,q}^2= \displaystyle\oplus_{s+t=q} \text{Tor}_p^R(H_s(A_{\ast}),H_t(C_{\ast})) \Rightarrow H_{p+q}(A_{\ast} \otimes_R C_{\ast}).
\]
\end{theorem}
For our purposes, $R= \Z[\Z \times \Gamma]$, $A_{\ast} = C_*(X,\Sigma \times \left\{0\right\};\Z[\Z \times \Gamma])$, and 
\[
C_{\ast}= \ldots \rightarrow 0 \rightarrow 0 \rightarrow \Z\left[\Z \times \frac{\Gamma}{[\Gamma,\Gamma]}\right] \rightarrow 0.
\]
Thus, $A_{\ast} \otimes_R C_{\ast} = C_*\left(X,\Sigma \times \left\{0\right\};\Z\left[\Z \times \frac{\Gamma}{[\Gamma,\Gamma]}\right]\right)$. Observe that 
\[
\text{E}_{p,q}^2=  \text{Tor}_p^{\Z[\Z \times \Gamma]}\left(H_q(X,\Sigma \times \left\{0\right\};\Z[\Z \times \Gamma]),\Z\left[\Z \times \frac{\Gamma}{[\Gamma,\Gamma]}\right]\right)
\]
since $H_i(C_{\ast}) \cong 0$ for $i \geq 1$. For $k$ large enough,
\[
H_1\left(X,\Sigma \times \left\{0\right\};\Z\left[\Z \times \frac{\Gamma}{[\Gamma,\Gamma]}\right]\right) \cong E_{1,0}^k \oplus E_{0,1}^k.
\]
The differential is given by $\text{d}^r: \text{E}_{p,q}^r \rightarrow E_{p-r,q+r-1}^r$ and
\[
E_{p,q}^{r+1}=\frac{\Ker d^r : E_{p,q}^r \rightarrow E_{p-r,q+r-1}^r}{\Im d^r : E_{p+q,q-r+1}^r \rightarrow E_{p,q}^r}.
\]
Since $H_0\left(X,\Sigma \times \left\{0\right\};\Z\left[\Z \times \Gamma\right]\right)\cong 0$, it follows that $E_{1,0}^k \cong 0$ for all $k \geq 2$. Note that $E_{0,1}^2=H_1(X,\Sigma \times \left\{0\right\};\Z[\Z \times \Gamma]) \otimes_{\Z[\Z \times \Gamma]} \Z\left[\Z \times \frac{\Gamma}{[\Gamma,\Gamma]}\right]$ and 
\[
\text{E}_{0,1}^3=\frac{\Ker (\text{E}_{0,1}^2 \rightarrow \text{E}_{-2,2}^2=0)}{\Im(\text{E}_{1,0}^2=0 \rightarrow \text{E}_{0,1}^2)}\cong \frac{E_{0,1}^2}{\left\{0\right\}}\cong E_{0,1}^2.
\]
Consequently, the sequence stabilizes. That is, for all $k \geq 2$, $E_{0,1}^k\cong E_{0,1}^2$. Therefore,
\[
H_1\left(X,\Sigma \times \left\{0\right\};\Z\left[\Z \times \frac{\Gamma}{[\Gamma,\Gamma]}\right]\right)  
\cong H_1(X,\Sigma \times \left\{0\right\};\Z[\Z \times \Gamma]) \otimes_{\Z[\Z \times \Gamma]} \Z\left[\Z \times \frac{\Gamma}{[\Gamma,\Gamma]}\right].
\]

\section{Concordance of Knots in a Thickened Surface}

Previous work \cite{stableequiv, Cobord} has defined (virtual) concordance of knots in a thickened surface. In this paper, we use a more restrictive notion of concordance. 
\begin{definition}
Let $K_0, K_1$ be knots in a thickened surface $\Sigma \times I$. We say that the two knots $K_0$ and $K_1$ are \textit{smoothly concordant} if there exists a smooth embedding $f:S^1 \times I \rightarrow \Sigma \times I \times I$ such that
\[
f\left(S^1 \times \left\{0\right\}\right)= K_0 \times \left\{0\right\} \subset \Sigma \times I \times \left\{0\right\}
\]
and
\[
f\left(S^1 \times \left\{1\right\}\right)= K_1 \times \left\{1\right\} \subset \Sigma \times I \times \left\{1\right\}
\]
That is, $K_0$ and $K_1$ cobound a smoothly embedded cylinder in $\Sigma \times I \times I$. If $K$ is concordant to the unknot, we call $K$ a slice knot. 
\end{definition}
The proof of the Fox-Milnor Theorem investigates the slice disk complement, as well as the knot complement. The proof of our result will follow a similar strategy, so we prove some preliminary results regarding the knot complement and the concordance complement. 

\subsection{The Knot Complement}

In order to work with the Alexander polynomial we define, we need to understand the homology of the knot complement. 
\begin{lemma}\label{homofcomp}
Let $K$ be a knot in a thickened surface $\Sigma \times I$, where $g=\text{genus}(\Sigma)$. Then,
\begin{displaymath}
   H_i\left(\Sigma \times I \backslash K\right) \cong \left\{
     \begin{array}{lr}
     	 \Z & : i=0\\
       \langle t \rangle \oplus \Z^{2g} & : i=1\\
       \Z^2 & : i=2\\
       0 & : i>2
     \end{array}
   \right.
\end{displaymath} 
where $t$ represents a meridian of the knot.
\end{lemma}
\noindent This can be proven using the Mayer Vietoris exact sequence where $U=N(K)$ is a tubular neighborhood of $K$ in $\Sigma \times I$, $V= \Sigma \times I - \text{Int}(N(K))$, and $U \cup V = \Sigma \times I$. One important point in the proof of \ref{homofcomp} is that there is a canonical isomorphism $H_1(\Sigma \times I \backslash K) \cong \Z \oplus H_1(\Sigma \times I)$. From the Mayer Vietoris exact sequence, we have 
\[
H_1\left(\partial N(K)\right) \rightarrow H_1\left(N(K)\right) \oplus H_1\left(\Sigma \times I \backslash \text{int}N(K)\right) \rightarrow H_1\left(\Sigma \times I \right) \rightarrow H_{0}\left(\partial N(K)\right)
\]  
which gives rise to the short exact sequence
\[
0 \xrightarrow[]{} \Z \xrightarrow[]{n \mapsto n \cdot [\mu_K]} H_1(\Sigma \times I \backslash K) \xrightarrow[]{} H_1(\Sigma \times I)\xrightarrow[]{} \Z
\]
where $\mu_K$ is a meridian of the knot. Since $H_1(\Sigma \times I)$ is a free abelian group, this short exact sequence splits. Moreover, there is a canonical splitting given by the map
\[
H_1(\Sigma \times I) \xleftarrow[]{\cong} H_1(\Sigma \times \{0\}) \xrightarrow[]{} H_1(\Sigma \times I \backslash K)
\]
which means that there is a canonical isomorphism $H_1(\Sigma \times I \backslash K) \cong \Z \oplus H_1(\Sigma \times I)$. This will be of importance when we compare the Alexander polynomials of knots in a thickened surface. 

\noindent Given our definition of the Alexander module, we will also need to understand the homology of the pair $(\Sigma \times I \backslash K, \Sigma \times \left\{0\right\})$.
\begin{lemma}\label{homofcomppair}
Let $K$ be a knot in a thickened surface $\Sigma \times I$. Then,
\begin{displaymath}
   H_i(\Sigma \times I \backslash K, \Sigma \times \left\{0\right\}) = \left\{
     \begin{array}{lr}
     	 \Z & : i=1,2\\
       0 & : i \neq 1, 2
     \end{array}
   \right.
\end{displaymath} 
\end{lemma} 
\noindent This result can be proven by using the long exact sequence in homology of the pair. 

\subsection{The Concordance Complement} The proof of our main result will require us to understand the complement of the concordance. We will compute its homology and determine its boundary. Let $C$ denote the image of the concordance $f$ and let $N(C)$ be a neighborhood of the cylinder in $\Sigma \times I \times I$. 
\begin{lemma}
Let $K_0, K_1$ be concordant knots in a thickened surface $\Sigma \times I$ and let $C$ be the cylinder in $\Sigma \times I \times I$ cobounded by the knots. Then 
\begin{displaymath}
   H_i(\Sigma \times I \times I \backslash N(C)) = \left\{
     \begin{array}{lr}
     	 \Z & : i=0\\
       \Z \oplus \Z^{2g} & : i=1\\
       \Z^2 & : i=2\\
       0 & : i>2
     \end{array}
   \right.
\end{displaymath}
where $N(C)$ is a neighborhood of $C$ in $\Sigma \times I \times I$ and $g=$ genus($\Sigma$).
\end{lemma}
\noindent This result can be proven using the Mayer Vietoris exact sequence, where $U = \overline{\Sigma \times I \times I \backslash N(C)}$, $V=N(C)$, and $\Sigma \times I \times I = U \cup V$. We also need to know the homology of the pair $(\Sigma \times I \times I \backslash N(C),\Sigma \times \left\{0\right\} \times \left\{0\right\})$.
\begin{lemma}
Let $K_0, K_1$ be concordant knots in a thickened surface $\Sigma \times I$ and let $C$ be the cylinder in $\Sigma \times I \times I$ cobounded by the knots. Then 
\begin{displaymath}
   H_i(\Sigma \times I \times I \backslash N(C), \Sigma \times \left\{0\right\} \times \left\{0\right\}) = \left\{
     \begin{array}{lr}
     	 \Z & : i=1,2\\
       0 & : i \neq 1, 2
     \end{array}
   \right.
\end{displaymath} 
where $N(C)$ is a neighborhood of $C$ in $\Sigma \times I \times I$ and $g=$ genus($\Sigma$).
\end{lemma} 
\noindent This result can be proven using the long exact sequence of the pair. 

The proof of the Fox-Milnor Theorem showed that the boundary of the slice disk complement is the 0-surgery of $S^3$ along the knot. Then, the Milnor Torsion of this pair was related to the Milnor Torsion of the knot complement. In particular, it was necessary to demonstrate the 0-surgery of $S^3$ along the knot $K$ did not change the Alexander module in such a way that the Alexander polynomial is changed. We will be following a similar strategy, so we will prove analogous results for knots in a thickened surface. 

First, we investigate the boundary of the concordance complement. Let $M=\Sigma \times I \times I \backslash N(C)$.  
Then,
\[
\partial M = \left(\partial\left(\Sigma \times I \times I\right)- \text{Int}\left(N(C) \cap \partial\left(\Sigma \times I \times I\right)\right)\right) \cup \left(\partial N(C) \cap \text{Int} \left(\Sigma \times I \times I\right)\right).
\]
A schematic of the concordance complemented is depicted in Figure \ref{concordcompsch}.

\begin{figure}[h]\begin{centering}
\includegraphics[scale=.15]{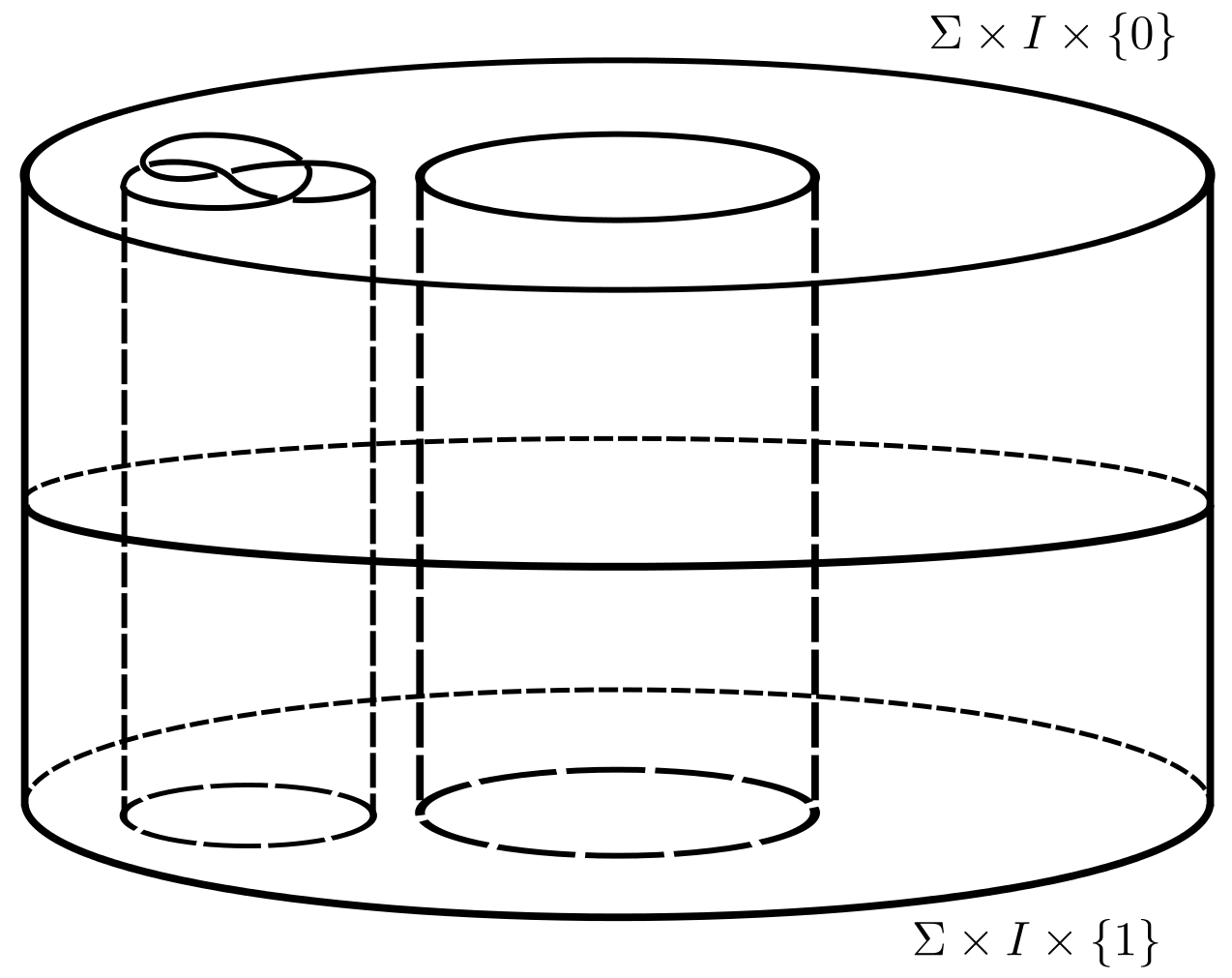}
\caption{Schematic of concordance complement}
\label{concordcompsch}
\end{centering}
\end{figure}
Write $\partial M = A \cup_h B$ where
\[
A=\partial\left(\Sigma \times I \times I\right)- \text{Int}\left(N(C) \cap \partial\left(\Sigma \times I \times I\right)\right)
\]
and
\[
B=\partial N(C) \cap \text{Int} \left(\Sigma \times I \times I\right)=S^1 \times I \times \partial I^2.
\]
Observe that,
\[
\partial \left(\Sigma \times I \times I\right)=\Sigma \times I \times \left\{0\right\} \sqcup\Sigma \times I \times \left\{1\right\} \sqcup \Sigma \times \left\{0\right\} \times I \sqcup \Sigma \times \left\{1\right\} \times I
\]
and
\[
\text{Int}\left(N(C) \cap \partial\left(\Sigma \times I \times I\right)\right)= N(K_0) \sqcup N(K_1)
\]
Therefore,
\[
A=X_0 \times \left\{0\right\} \sqcup X_1 \times \left\{1\right\} \sqcup \Sigma \times \left\{0\right\} \times I \sqcup \Sigma \times \left\{1\right\} \times I
\]
where $X_0 = \Sigma \times I \backslash N(K_0)$ and $X_1= \Sigma \times I \times \backslash N(K_1)$. The spaces $A$ and $B$ are glued together as follows. The longitude of $S^1 \times \left\{0\right\} \times \partial I^2$ is glued along the meridian of $\partial N(K_0)$ and the longitude of $S^1 \times \left\{1\right\} \times \partial I^2$ is glued along the meridian of $\partial N(K_1)$. Re-write $\partial M = M_0 \cup M_1$ where
\[
M_0= X_0 \times \left\{0\right\} \cup_h S^1 \times \left\{0\right\} \times \partial I^2
\]
and
\[
M_1= X_1 \times \left\{1\right\} \cup_h S^1 \times I \times \partial I^2 \cup \Sigma \times \left\{0\right\} \times I \cup \Sigma \times \left\{1\right\} \times I.
\]
The subspaces $X_0 \times \left\{0\right\} \cup_h S^1 \times \left\{0\right\} \times \partial I^2$ and $X_1 \times \left\{1\right\} \cup_h S^1 \times \left\{1\right\} \times \partial I^2$ are surgery on $\Sigma \times I$ along the knots $K_0$ and $K_1$. In order to have a coherent coefficient system for the triple $\left(M,M_0,X_0 \times \left\{0\right\}\right)$, we need to verify that the surgery of $\Sigma \times I$ along the knot has the same homology group as both the concordance complement and the knot complement. 
\begin{lemma}\label{surgsamehom}
Let $K$ be a knot in thickened surface $\Sigma \times I$ and let $X=\Sigma \times I \backslash N(K)$. Let
\[
X'=X\cup_h S^1 \times S^1
\]
where $h:S^1 \times S^1 \rightarrow \partial N(K)$ is the homeomorphism sending the longitude of $S^1 \times S^1$ to the meridian on $\partial N(K)$. Then, $H_i(X')\cong H_i(X)$ for all $i \geq 0$.
\end{lemma}
\begin{proof}
Note that 
\[
X-\text{Int}(X)=\Sigma\times \left\{0\right\} \sqcup \Sigma\times \left\{1\right\} \sqcup \partial N(K)
\]
and
\[
X'-\text{Int}(X)=\Sigma\times \left\{0\right\} \sqcup \Sigma\times \left\{1\right\} \sqcup Z
\]
where $Z=\partial N(K) \cup_h S^1 \times S^1$. By Excision,
\[
H_i(X',X) \cong H_i(\Sigma\times \left\{0\right\} \sqcup \Sigma\times \left\{1\right\} \sqcup Z,\Sigma\times \left\{0\right\} \sqcup \Sigma\times \left\{1\right\} \sqcup \partial N(K))
\]
In the long exact sequence of the pair
\[
\left(\Sigma\times \left\{0\right\} \sqcup \Sigma\times \left\{1\right\} \sqcup Z,\Sigma\times \left\{0\right\} \sqcup \Sigma\times \left\{1\right\} \sqcup \partial N(K)\right)
\]
the homomorphism 
\[
H_i\left(\Sigma\times \left\{0\right\} \sqcup \Sigma\times \left\{1\right\} \sqcup \partial N(K)\right) \rightarrow H_i\left(\Sigma\times \left\{0\right\} \sqcup \Sigma\times \left\{1\right\} \sqcup Z\right)
\]
is an isomorphism. Therefore,
\[
H_i(\Sigma\times \left\{0\right\} \sqcup \Sigma\times \left\{1\right\} \sqcup Z,\Sigma\times \left\{0\right\} \sqcup \Sigma\times \left\{1\right\} \sqcup \partial N(K)) \cong 0.
\]
and thus $H_i(X',X) \cong 0$. This implies that the homomorphism
\[
H_i(X) \rightarrow H_i(X')
\]
from the long exact sequence of the pair $(X',X)$ is an isomorphism.
\end{proof}

\section{A Fox-Milnor Theorem for Knots in a Thickened Surface} In this section, we will state and prove a Fox Milnor Theorem for knots in a thickened surface. First, we state our main result. 
\begin{theorem}\label{mainresult1}
Let $K_0, K_1$ be concordant knots in a thickened surface $\Sigma \times I$. Then
\[
\Delta\left(K_0\right)= \alpha \overline{\alpha} \Delta\left(K_1\right)
\]
where $\alpha$ is an element of the field of fractions of $\Z[\Z \oplus H_1(\Sigma)]$.
\end{theorem}
\noindent We will prove this result by first showing that $\Delta(K)$ is the Alexander function of the pair $\left(X, \Sigma \times \{0\} \right)$. The Alexander function of the pair $\left(X, \Sigma \times \{0\} \right)$ is the same as the Milnor torsion of the pair $\left(X, \Sigma \times \{0\} \right)$ \cite{Turaev}, so we will then use results about Milnor torsion to prove our main result. 

\noindent The Alexander function of a pair $(X, Y)$ is an element of the field of fractions of $\Z[H_1(X)]$ defined as
\[
A(X, Y) = \prod_{i=0}^m \left[\Delta_0 \left(H_i(X, Y; \Z[H_1(X)])\right)\right]^{(-1)^{i+1}}.
\]
We demonstrate that $A\left(X, \Sigma \times \left\{0\right\}\right) = \Delta(K)$ where $K$ is a knot in a thickened surface and $X=\Sigma \times I - K$. First, note that 
\[
H_i\left(X,\Sigma \times \left\{0\right\};\Z\left[H_1(X)\right]\right)\cong 0
\]
for $i \neq 1,2$ since $X$ is connected and we can collapse all 3-cells to the boundary. Thus,
\[
A\left(X, \Sigma \times \left\{0\right\}\right) = \Delta_0 \left( H_1\left(X,\Sigma \times \left\{0\right\};\Z\left[H_1(X)\right]\right)\right) \cdot \Delta_0 \left( H_2\left(X,\Sigma \times \left\{0\right\};\Z\left[H_1(X)\right]\right)\right)^{-1}.
\]
The matrix representing the boundary map
\[
C_2\left(X,\Sigma \times \left\{0\right\};\Z\left[H_1(X)\right]\right) \stackrel{\partial_2}{\rightarrow} C_1\left(X,\Sigma \times \left\{0\right\};\Z\left[H_1(X)\right]\right)
\]
is a presentation matrix for $H_1\left(X,\Sigma \times \left\{0\right\};\Z\left[H_1(X)\right]\right)$. If $\partial_2$ is not injective, then $\text{rank } H_2\left(X,\Sigma \times \left\{0\right\};\Z\left[H_1(X)\right]\right) \neq 0$ and therefore $\Delta_0 \left( H_2\left(X,\Sigma \times \left\{0\right\};\Z\left[H_1(X)\right]\right)\right)=0$. Consequently, $A\left(X, \Sigma \times \left\{0\right\}\right)=0$. But,
\begin{align*}
\chi\left(X, \Sigma \times \left\{0\right\}\right)=0=\text{rank } H_2\left(X,\Sigma \times \left\{0\right\};\Z\left[H_1(X)\right]\right) - \text{rank } H_1\left(X,\Sigma \times \left\{0\right\};\Z\left[H_1(X)\right]\right).
\end{align*}
Thus,
\[
\text{rank } H_2\left(X,\Sigma \times \left\{0\right\};\Z\left[H_1(X)\right]\right) =\text{rank } H_1\left(X,\Sigma \times \left\{0\right\};\Z\left[H_1(X)\right]\right)
\]
Thus, in the case that $\text{rank } H_2\left(X,\Sigma \times \left\{0\right\};\Z\left[H_1(X)\right]\right) \neq 0$, it follows that 
\[
\text{rank } H_1\left(X,\Sigma \times \left\{0\right\};\Z\left[H_1(X)\right]\right) \neq 0.
\]
Therefore,
\[
A\left(X, \Sigma \times \left\{0\right\}\right)=0 = \Delta(K)
\]
If $\partial_2$ is injective, then
\[
H_2\left(X,\Sigma \times \left\{0\right\};\Z\left[H_1(X)\right]\right)\cong 0
\]
which implies that 
\[
\Delta_0 \left(H_2\left(X,\Sigma \times \left\{0\right\};\Z\left[H_1(X)\right]\right)\right)=1
\]
Hence,
\[
A\left(X, \Sigma \times \left\{0\right\}\right)=\Delta_0 \left(H_1\left(X,\Sigma \times \left\{0\right\};\Z\left[H_1(X)\right]\right)\right) = \Delta(K).
\] 

The proof of the classical Fox-Milnor Theorem uses results relating the Milnor torsion of the triple consisting of the slice disk complement, the 0-surgery of $S^3$ along the knot, and the knot complement. We will follow a similar strategy by first stating and proving some preliminary results which will relate the torsion of the spaces in the triple which includes the concordance complement, the surgery of $\Sigma \times I$ along the knot, and the knot complement. First we establish the following.

\begin{lemma} \label{zerosurgery} 
Let $K$ be a knot in thickened surface $\Sigma \times I$ and let $X=\Sigma \times I \backslash N(K)$. Let
\[
X'=X\cup_h S^1 \times S^1
\]
where $h:S^1 \times S^1 \rightarrow \partial N(K)$ is the homeomorphism sending the longitude of $S^1 \times S^1$ to the meridian on $\partial N(K)$. Then, $A(X, \Sigma \times \left\{0\right\}) = A(X', \Sigma \times \left\{0\right\})$. Note that this also implies that $\tau(X, \Sigma \times \left\{0\right\}) = \tau(X', \Sigma \times \left\{0\right\})$.
\end{lemma}
\begin{proof}
By the same excision argument as in Lemma \ref{surgsamehom}, except with coefficients in $\Z[H_1(X')]$, we have that $H_1(X',X; \Z[H_1(X')]) \cong 0$. This implies that the homomorphism 
\[
 H_i(X, \Sigma \times \left\{0\right\}; \Z[H_1(X)]) \stackrel{\cong}{\rightarrow} H_i(X', \Sigma \times \left\{0\right\}; \Z[H_1(X)])
\]
from the lonq exact sequence of the triple $(X',X, \Sigma \times \left\{0\right\})$ is an isomorphism for all $i \geq 0$. Consequently, 
\[
\Delta_0 H_i(X, \Sigma \times \left\{0\right\}; \Z[H_1(X')])= \Delta_0 H_i(X', \Sigma \times \left\{0\right\}; \Z[H_1(X')])
\]
for all $i\geq 0$. Therefore 
\[
A(X, \Sigma \times \left\{0\right\}) = A(X', \Sigma \times \left\{0\right\}).
\]
\end{proof}

We will need some results about Milnor torsion. The first provides a relation between the Milnor torsion of pairs of spaces that come from a triple. 
\begin{lemma}\cite{Turaev} \label{tortrip}
Suppose $Z \subset Y \subset X$ is a triple of topological spaces. Let $\phi: \Z[H_1(X)] \rightarrow R$ be a ring homomorphism and let $\psi: \Z[H_1(Y)] \rightarrow \Z[H_1(X)]$ be the homomorphism induced by inclusion. If $\tau^{\phi}(X,Y) \neq 0$ or $\tau^{\phi \circ \psi}(Y,Z) \neq 0$, then
\[
\tau^{\phi}(X,Z)=\tau^{\phi}(X,Y)\tau^{\phi \circ \psi}(Y,Z).
\]
\end{lemma}
The next result is a duality theorem for Milnor torsion. 
\begin{theorem} \label{tdual}
Let $W$ be an $n$-manifold such that $\partial W = M_1 \cup M_2$ and $M_1 \cap M_2 = \partial M_1 = \partial M_2$. Then,
\[
\tau(W, M_1) = \overline{\tau(W,M_2)}^{(-1)^{n+1}}.
\]  
\end{theorem}
\noindent Note that we follow the convention set in \cite{Turaev}; we let $\tau(X,Y)$ denote $\tau^{\varphi}(X,Y)$ where $\varphi: \Z[H_1(X)] \rightarrow \Z[G]$ and $G = H_1(X)/\text{Tors}H_1(X)$. The proof of this result is very similar to the proofs of other duality results (see \cite{TPD, TTI}).

\noindent We are now ready to prove the main result. Suppose $K_0$ and $K_1$ are concordant knots and let $X_0 = \Sigma \times I - N(K_0)$ and $X_1 = \Sigma \times I - N(K_1)$. Let $M= \Sigma \times I \times I - N(C)$ be the complement of the concordance between $K_0$ and $K_1$. Write $\partial M = M_0 \cup M_1$ as in the previous section. Then, by \ref{zerosurgery},
\[
\Delta(K_0) = A\left(X_0 \times \left\{0\right\}, \Sigma \times \left\{0\right\} \times \left\{0\right\}\right) =A\left(M_0, \Sigma \times \left\{0\right\} \times \left\{0\right\}\right)
\]
and thus
\[
\Delta(K_0) = \tau\left(M_0, \Sigma \times \left\{0\right\} \times \left\{0\right\}\right).
\]
We will apply \ref{tortrip} to the triple
\[
\left(M,M_0,\Sigma \times \left\{0\right\} \times \left\{0\right\}\right),
\]
so first we will verify that $\tau(M,M_0) \neq 0$. By \cite{cot}, it suffices to show the following. Consider a pair of CW-complexes $(X,Y)$. Let $\varphi: \pi_1(X) \rightarrow H$ where $H$ is a free abelian group. If $H_{\ast}(X,Y) =0$, then $H_{\ast}(X,Y;\Q(H)) =0$, where $\Q(H)$ is the quotient field of $\Z[H]$. This implies that $\tau^{\varphi}(X,Y) \neq 0$. Thus, we will establish that $H_{\ast}(M,M_0)=0$. Consider the long exact sequence in homology of the pair $(M,M_0)$. Since $H_{\ast}(M_0) \xrightarrow[]{\cong} H_{\ast}(M)$ is an isomorphism, it follows that $H_{\ast}(M,M_0)= 0$. Applying \ref{tortrip} to the triple 
\[
\left(M,M_0,\Sigma \times \left\{0\right\} \times \left\{0\right\}\right)
\]
we have,
\[
\Delta(K_0) = \tau\left(M_0, \Sigma \times \left\{0\right\} \times \left\{0\right\}\right)=\tau\left(M,M_0\right)^{-1} \tau\left(M,\Sigma \times \left\{0\right\} \times \left\{0\right\}\right).
\] 
Since Milnor torsion is invariant under homotopy equivalence, 
\[
\tau\left(M,\Sigma \times \left\{0\right\} \times \left\{0\right\}\right)=\tau\left(M,\Sigma \times \left\{0\right\} \times \left\{I\right\}\right)=\tau\left(M,\Sigma \times \left\{0\right\} \times \left\{1\right\}\right)
\]
and consequently
\[
\Delta(K_0) =\tau\left(M,M_0\right)^{-1} \tau\left(M,\Sigma \times \left\{0\right\} \times \left\{1\right\}\right).
\] 
The same argument as above shows that we can apply \ref{tortrip} to the triple 
\[
\left(M,M_1,\Sigma \times \left\{0\right\} \times \left\{1\right\}\right).
\]
Therefore, we have
\[
\Delta(K_0) =\tau\left(M,M_0\right)^{-1} \tau\left(M,\Sigma \times \left\{0\right\} \times \left\{1\right\}\right)=\tau\left(M,M_0\right)^{-1} \tau\left(M,M_1\right) \tau\left(M_1,\Sigma \times \left\{0\right\} \times \left\{1\right\}\right).
\] 
By \ref{tdual},
\begin{align*}
\Delta(K_0) &=\tau\left(M,M_0\right)^{-1} \tau\left(M,M_1\right) \tau\left(M_1,\Sigma \times \left\{0\right\} \times \left\{1\right\}\right) \\ &=\tau\left(M,M_0\right)^{-1} \overline{\tau\left(M,M_0\right)^{-1}} \tau\left(M_1,\Sigma \times \left\{0\right\} \times \left\{1\right\}\right) .
\end{align*}
Note that
\begin{align*}
M_1 &= X_1 \times \left\{1\right\} \cup_h S^1 \times I \times \partial I^2 \cup \Sigma \times \left\{0\right\} \times I \cup \Sigma \times \left\{1\right\} \times I\\
&\simeq  X_1 \times \left\{1\right\} \cup_h S^1 \times I \times \partial I^2 \cup \Sigma \times \left\{0\right\} \times \left\{1\right\} \cup \Sigma \times \left\{1\right\} \times \left\{1\right\}\\
&= X_1 \times \left\{1\right\} \cup_h S^1 \times I \times \partial I^2\\
&= X_1' \times \left\{1\right\}  
\end{align*}
where $X_1'$ is the surgery of $\Sigma \times I$ along the knot $K_1$. Thus,
\begin{align*}
\Delta(K_0) &=\tau\left(M,M_0\right)^{-1} \overline{\tau\left(M,M_0\right)^{-1}} \tau\left(M_1,\Sigma \times \left\{0\right\} \times \left\{1\right\}\right)\\  & =\tau\left(M,M_0\right)^{-1} \overline{\tau\left(M,M_0\right)^{-1}} \tau\left(X_1' \times \left\{1\right\},\Sigma \times \left\{0\right\} \times \left\{1\right\}\right).
\end{align*}
By \ref{zerosurgery},
\begin{align*}
\Delta(K_0) & =\tau\left(M,M_0\right)^{-1} \overline{\tau\left(M,M_0\right)^{-1}} \tau\left(X_1' \times \left\{1\right\},\Sigma \times \left\{0\right\} \times \left\{1\right\}\right)\\   & =\tau\left(M,M_0\right)^{-1} \overline{\tau\left(M,M_0\right)^{-1}} \tau\left(X_1 \times \left\{1\right\} ,\Sigma \times \left\{0\right\} \times \left\{1\right\}\right).
\end{align*}
Lastly,
\begin{align*}
\Delta(K_0) & =\tau\left(M,M_0\right)^{-1} \overline{\tau\left(M,M_0\right)^{-1}} \tau\left(X_1 \times \left\{1\right\} ,\Sigma \times \left\{0\right\} \times \left\{1\right\}\right)\\  & =\alpha \overline{\alpha} A\left(X_1,\Sigma \times \left\{0\right\}\right)\\ & = \alpha \overline{\alpha} \Delta(K_1)
\end{align*}
where $\alpha = \tau\left(M,M_0\right)^{-1}$.

\bibliographystyle{amsplain}

\end{document}